\documentclass[a4paper,12pt]{amsart}
\usepackage[utf8]{inputenc}
\usepackage[T1]{fontenc}
\usepackage[psamsfonts]{euscript}

\usepackage{amsmath}
\usepackage{stmaryrd}
\usepackage{bbm}
\usepackage{mathrsfs}
\usepackage{tikz}
\usepackage{enumerate}

\usepackage{natbib}

\usepackage[bbgreekl]{mathbbol}
\usepackage{mathtools,halloweenmath}

\usepackage{txfonts}

\usepackage[margin=3cm]{geometry}

\usepackage{hyperref}

\usepackage{xcolor}
\hypersetup{
    colorlinks,
    linkcolor={red!90!black},
    citecolor={green!70!black},
    urlcolor={blue!80!black}
}

\usepackage{todonotes}

\newtheorem{theorem}{Theorem}
\newtheorem{lemma}{Lemma}
\newtheorem{corollary}{Corollary}
\newtheorem{fact}{Fact}

\theoremstyle{remark}
\theoremstyle{definition}

\renewcommand\iff{\longleftrightarrow}

\newcommand\iffdef{\stackrel{\text{\textup{df}}}{\longleftrightarrow}}

\newcommand\defeq{\coloneqq}

\newcommand{\power}{\EuScript{P}}
\newcommand{\powerne}{\power_{\!+}}

\newcommand\rarrow{\rightarrow}
\newcommand\larrow{\leftarrow}

\DeclareMathOperator{\ingr}{\sqsubseteq}
\DeclareMathOperator{\ingrs}{\ingr_{\mathrm{S}}}
\DeclareMathOperator{\ningr}{\nsqsubseteq}
\DeclareMathOperator{\ningrs}{\ningr_{\mathrm{S}}}

\DeclareMathOperator{\ext}{\Lbag}  
\DeclareMathOperator{\overl}{\bigcirc}
\DeclareMathOperator{\overls}{\overl_{\mathrm{S}}}
\newcommand{\MS}{\mathbf{MS}}

\newcommand\calA{\mathcal{A}}

\newcommand{\+}{\mathrel{\mathsf{S}}} %you can get proper mathematical spacing like this

\makeatletter
\newcommand*{\rom}[1]{\expandafter\@slowromancap\romannumeral #1@}
\makeatother

\newcommand{\Sumd}{\+_{\ingr}}

\newcommand{\poverl}{\overl_{\mathord{\ingr}}}
\newcommand{\pext}{\ext_{\mathord{\ingr}}}

\newcommand{\soverl}{\overl_{\mathord{\+}}} %overlap for S
\newcommand{\sext}{\ext_{\mathord{\+}}} %incompatibility for S

\DeclareMathOperator{\Ingr}{\mathbb{P}}
\newcommand{\Ingrs}{\Ingr_{\mathord{\+}}} %all S-parts
\newcommand{\ingrd}{\ingr_{\+}}
\let\ingrs=\ingrd
\newcommand{\Sigmas}{\Sigma_{\mathord{\+}}}
\newcommand\Sums{\mathrel{\+_{\mathord{\ingrs}}}}

\renewcommand{\MS}{\mathbf{M}}
\renewcommand{\ext}{\mathrel{\bot}}

\newcommand{\bfS}{\mathbf{S}}

{\medskip\color{magenta}{\noindent RG: \  }}{\hfill{$\square$}\color{black}\par\medskip}

\title{The sum relation as a primitive concept of mereology}

\author{Rafa\l\ Gruszczy\'nski, Dazhu Li}

\address{Rafa\l\ Gruszczy\'nski\\
Department of Logic\\
Nicolaus Copernicus University in Toru\'n\\
Poland\\
\textsc{Orcid:} 0000-0002-3379-0577}
\email{gruszka@umk.pl}

\urladdr{www.umk.pl/\textasciitilde gruszka}

\address{Dazhu Li\\
Institute of Philosophy\\
Chinese Academy of Sciences\\
Beijing, China;
Department of Philosophy\\
University of Chinese Academy of Sciences\\
Beijing, China\\
\textsc{Orcid:} 0000-0003-4780-1705}
\email{lidazhu@ucas.ac.cn}

\begin{document}

\maketitle

\begin{abstract}
Mereology in its formal guise is usually couched in a language whose signature contains only one primitive binary predicate symbol representing the \emph{part of} relation, either the proper or improper one. In this paper, we put forward an approach to mereology that uses \emph{mereological sum} as its primitive notion, and we demonstrate that it is definitionally equivalent to the standard parthood-based theory of mereological structures.

    \medskip

\noindent MSC: 03A05

\medskip

\noindent Keywords: mereology, mereological sum, mereological axioms, definitional equivalence of mereological theories, theory of relations
\end{abstract}

\section{Introduction}\label{sec:motivation}

`Mereology' is a catch-all term encompassing the whole menagerie of theories of parts and whole, both formal and informal. As formal theories, various systems of mereology usually share a common language whose signature contains only one primitive binary predicate symbol representing the \emph{part of} relation. Various axioms are applied to it, leading to a variety of theories that represent different philosophical perspectives.

However, \emph{parthood} is not the only choice among the primitives of mereology. \emph{Disjointness} is taken as basic by \cite{Leonard-Goodman-TCOIAIU}, and \emph{overlap}---the complement of disjointness---by \cite{Goodman-SP}.\footnote{See \cite{Parsons-TMPOM}, \cite[Chapter 2.4]{Cotnoir-Varzi-M} and \cite[Chapter IV]{Pietruszczak-M-eng} for more on the axiomatizations of mereology with these primitives.} The choice of the primitive symbol often reflects different philosophical perspectives, particularly concerning metaphysics, and indeed, can play a pivotal role in resolving many philosophical debates (see e.g., \citealp{Varzi-M}).

Among the notions of mereology, that of a \emph{mereological sum} stands out as one of the most intriguing, and probably most controversial. As defined within a mereological system in terms of one of the aforementioned primitives the sum relation has been examined extensively from the perspectives of both logic and philosophy.\footnote{For the former analysis see \citep{Cotnoir-Varzi-M}, \citep{Gruszczynski-Pietruszczak-HTDMS}, \citep{Gruszczynski-MFAAUB}, \citep{Gruszczynski-Pietruszczak-RSMSPS}, \citep{Pietruszczak-AGCOBAPOAW}, \citep{Hovda-WICM}, \citep{Pietruszczak-M-eng}, \citep{Pietruszczak-FTP}, for the latter \citep{Cotnoir-Varzi-M}, \citep{Cotnoir-Baxter-CAI}, \citep{Varzi-M}, to mention a few.} Despite this, no one has so far put forward a system of mereology in which mereological sum is a primitive concept. We intend to fill this gap in the realm of mereological investigations. To be more precise, we put an axiomatic system whose primitives are `$\+$' interpreted as mereological sum, and the standard set-theoretical `$\in$'.  Our approach in model theoretical, that is we put in focus structures $\langle M,\+\rangle$ where $M$ is the domain and $\+$ is a binary (hybrid) relation in the Cartesian product of $M$ and the power set of~$M$. We propose a set of axioms strong enough to show that the structures satisfying them are definitionally equivalent to and in one-to-one correspondence with mereological structures $\langle M,\mathord{\ingr}\rangle$ in the sense of \cite{Pietruszczak-PM}. In the event, we prove that the \emph{mereological sum} concept can indeed replace the binary \emph{part of} relation as a primitive notion of mereology.

 The paper is organized as follows.  In Section~\ref{sec:parthood-mereology}, we briefly overview the theory of mereological structures formulated with a primitive binary relation of parthood. In Section \ref{sec:axioms-for-sum} we formulate axioms for mereology based on the notion of \emph{mereological sum}, and we establish basic properties of \emph{sum structures}. In particular, we demonstrate that the \emph{part of} relation is definable in the framework of sum structures and satisfies all standard axioms for mereological structures. Section~\ref{sec:sum-axioms-in-MS} contains the proof that all sum axioms hold in mereological structures, with the sum defined in the standard way. Together with the results from Section~\ref{sec:axioms-for-sum} this shows that the two approaches are definitionally equivalent and that there is a one-to-one correspondence between mereological and sum structures. In Section~\ref{sec:independence} we prove the independence of the set of axioms for the sum relation, and in Section~\ref{sec:summary} we take the stock of the paper outcomes and put forward problems for future work.

 In the paper, we use standard logical notation. Moreover, we abbreviate `$\neg\exists$' with `$\nexists$'. For any set $X$, $\power(X)$ is the set of all subsets of $X$, and $\powerne(X)$ the set of all its \emph{non-empty} subsets.

\section{Mereological structures}\label{sec:parthood-mereology}

In this section, we concisely review the definition of the class $\MS$ of \emph{mereological structures} that consists of pairs of the form $\langle M,\mathord{\ingr}\rangle$, where $M$ is a non-empty class of {\em objects} and $\mathord{\ingr}\subseteq M\times M$ is a binary relation on $M$ called \emph{part of\/} relation and satisfies the following `$\ingr$-axioms' (`$x\ningr y$' abbreviates `$\neg x\ingr y$'):
\begin{gather}\allowdisplaybreaks
x\ingr x\,,\tag{P1}\label{P1}\\
x\ingr y\wedge y\ingr x\rarrow x=y\,,\tag{P2}\label{P2}\\
x\ingr y\wedge y\ingr z\rarrow x\ingr z\,,\tag{P3}\label{P3}\\
x\ningr y\rarrow(\exists z\in M)\,\bigl(z\ingr x\wedge(\nexists u\in M)\, (u\ingr z\wedge u\ingr y)\bigr)\,,\tag{P4}\label{P4}\\
\begin{split}
(\forall{X\in\powerne(M)})&(\exists{x\in M})\,\bigl((\forall y\in X)\,y\ingr x\wedge{}\\
&(\forall a\in M)\,(a\ingr x\rarrow(\exists y\in X)(\exists z\in M)\,(z\ingr y\wedge z\ingr a))\bigr)\,.
\end{split}\tag{P5}\label{P5}
\end{gather}
 %where $\powerne(M)\defeq\power(M)\setminus\{\emptyset\}$.

 The axioms \eqref{P1}-\eqref{P3} require  the relation $\ingr$ to be a  partial order, i.e. it is reflexive, transitive and  anti-symmetric. The axiom \eqref{P4} states  if an object $x$ is not part of $y$, then $x$ has a part $z$ that has no common parts with $y$. Finally, \eqref{P5} reads  for any non-empty collection $X$ of objects, there exists an object $x$ such that (i) all objects in $X$ are parts of $x$ and (ii) any part of $x$ has common parts with some object in $X$.

 The axioms \eqref{P4} and \eqref{P5} can be made more perspicuous with the help of the following derived notions of, respectively, \emph{overlap} and \emph{disjointness} (or \emph{incompatibility}):
 \begin{align}
 x\poverl y&{}\iffdef(\exists z\in M)\,(z\ingr x\wedge z\ingr y)\,,\tag{$\mathrm{df}\,\mathord{\poverl}$}\label{df:overl}\\
x\pext y&{}\iffdef(\nexists z\in M)\,(z\ingr x\wedge z\ingr y)\,,\tag{$\mathrm{df}\,\mathord{\pext}$}\label{df:ext}
\end{align}
and the standard notion of \emph{mereological sum}, which is a hybrid relation between elements of the domain and its subsets:
\begin{equation}\tag{$\mathrm{df}\,\mathord{\Sumd}$}\label{df:Sum}
   x\Sumd X\iffdef (\forall y\in X)\,y\ingr x\wedge (\forall a\in M)\,(a\ingr x\rarrow(\exists y\in X)\,a\poverl y)\,.
\end{equation}
With these we can express \eqref{P4} and \eqref{P5} as:
\begin{align}
    x\ningr y\rarrow(\exists z\in M)\,(z\ingr x\wedge z\pext y)\,,\tag{P4$'$}\label{P4'}\\
    (\forall X\in\powerne(M))(\exists x\in M)\,x\Sumd X\,.\tag{P5$'$}\label{P5'}
\end{align}
Usually \eqref{P4'} is called the \emph{strong supplementation principle} and \eqref{P5'} the \emph{unrestricted sum axiom}. The axiom system we employ is presented and analyzed in \citep{Pietruszczak-PM} and is one of the possible modern formulations of Le\'sniewski's mereology.\footnote{As is well known, the system is not independent as the reflexivity of the \emph{part-of} relation can be derived from transitivity and strong supplementation.}

Similarly to \eqref{df:overl} and \eqref{df:ext}, we introduce the following notions of {\em $\Sumd$-overlap} and {\em $\Sumd$-disjointness}:
\begin{align}
x\overl_{\Sumd} y&\:\iffdef (\exists X, Y\in\power(M))(x\Sumd X\wedge y\Sumd Y\wedge X\cap Y\neq\emptyset)\,,\tag{$\mathrm{df}\,\mathord{\overl_{\Sumd}}$}\label{df:overl-Sumd}\\
x\ext_{\Sumd} y&\:\iffdef (\forall X,Y\in \power(M)) (x\Sumd X\wedge y\Sumd Y\rarrow X\cap Y=\emptyset)\,.\tag{$\mathrm{df}\,\mathord{\ext_{\Sumd}}$}\label{df:ext-Sumd}
\end{align}
The intended interpretation behind $\overl_{\Sumd}$ is clear: $x$ and $y$  $\Sumd$-overlap if they sum collections that share at least one object. W.r.t. the relations between $\poverl$ and $\overl_{\Sumd}$, and $\pext$ and  $\ext_{\Sumd}$, we have the following:

\begin{fact}\label{fact:dfSum-dfOverl->ESums-overl}
For any $\langle M,\mathord{\ingr}\rangle\in\MS$, it holds that:
\begin{align}
    x\overl_{\Sumd} y &{}\iff x\poverl y\,,\\
    x\ext_{\Sumd} y& {}\iff x\pext y\,.
\end{align}
\end{fact}

\begin{proof}
We merely consider for the first equivalence, and the other is similar.

\smallskip

($\rarrow$) If $x\Sumd X$ and $y\Sumd  Y$ and $z\in X\cap Y$, then from \eqref{df:Sum} we obtain that $z\ingr x\wedge z\ingr y$. Therefore $x\poverl y$ by \eqref{df:overl}.

\smallskip

($\larrow$) Let $z\ingr x$ and $z\ingr y$. We define $X:=\{x,z\}$ and $Y:=\{y,z\}$.  By \eqref{P1}, $x\ingr x$ and $y\ingr y$. Then, with \eqref{df:Sum} we have $x\Sumd X$ and $y\Sumd Y$. Obviously, $X\cap Y\not=\emptyset$. So, by \eqref{df:overl-Sumd}, it holds that $x\overl_{\Sumd} y$, as desired. \qedhere

%To show that the reverse implication does not have to hold we construct the following structure. Let $M\defeq\{1,2,3,12,23\}$ and:\par
%\begin{center}
%\begin{tikzpicture}
%\path (1,1.5) coordinate (12);
%\path (3,1.5) coordinate (23);
%\fill (12) circle (2pt);%% (23) circle (2pt);
%\node[above] at (12) {12};
%\node[above] at (23) {23};
%\foreach \x [count=\xi] in {0,2}
%{
%\path (\x,0) coordinate (\xi) node[below=2pt] {\xi};
%\ifnum\x=2\fill (\x,0) circle (2pt);\else\draw(\x,0) circle (2pt);\fi
%\ifnum\x<4\draw[shorten <=2pt,shorten >=2pt] (\x,0) -- (12);\else\x\relax\fi
%%\ifnum\x>0\draw[shorten <=2pt,shorten >=2pt] (\x,0) -- (23);\else\x\relax\fi
%}
%\end{tikzpicture}
%\end{center}
%We have that $12\overl 2$, yet there is no set containing 2 of which 12 is a mereological sum.
\end{proof}

It was Le\'sniewski's idea that the notion of \emph{part} can be characterized by means of the notion of mereological sum in the following way:\footnote{For details and the discussion of Le\'sniewski's approach to collections and sets see \cite[Chapter 1]{Pietruszczak-M-eng}.}
\begin{center}
\begin{tabular}{rcl}
      $x$ is part of $y$ &iff &there is a collection of objects $X$ such that $x$ is among $X$-es \\
     & &  and  $y$ is a mereological sum of $X$.
\end{tabular}
\end{center}
Let us observe that for $\Sumd$ the above can be proven to obtain.

\begin{fact}\label{fa:ingr=ingrsingr}
    If $\langle M,\mathord{\ingr}\rangle\in\mathbf{M}$, then\/\textup{:}
    \[
        x\ingr y\iff(\exists X\in\power(M))\,(y\Sumd X\wedge x\in X)\,.
    \]
\end{fact}
\begin{proof}
    ($\rarrow$) If $x\ingr y$, then take $X\defeq\{y,x\}$. It is routine to check that $y\Sumd X$.

    \smallskip

    ($\larrow$) If $X$ is such that $y\Sumd X$, then directly from \eqref{df:Sum} we have that every element of $X$ is part of $y$. In particular $x\ingr y$, as $x\in X$ by assumption.
\end{proof}

Our idea in what follows is to take the equivalence above as a \emph{definition} of the \emph{part-of} relation with the notion of \emph{sum} as primitive. That is, with $\+$ being the primitive sum, we want $\ingrs$ to behave as the well-known parthood relation. As Fact~\ref{fa:ingr=ingrsingr} demonstrates, for $\ingr$ and $\Sumd$ we have:
\begin{equation}\label{eq:equality-of-parthoods}
    \mathord{\ingr}=\mathord{\ingr_{\Sumd}}\,.
\end{equation}
Thus, one of the particular problems for $\+$ to be solved is to find appropriate axiomatization for which the part of relation $\ingrs$ will be such that:
\begin{equation}\label{eq:equality-of-sums}
    \mathord{\+}=\mathord{\+_{\ingrs}}\,.
\end{equation}
\eqref{eq:equality-of-parthoods} and \eqref{eq:equality-of-sums} together will be useful in establishing a one-to-one correspondence between part-of relations and mereological sums.

The class of axioms \eqref{P1}--\eqref{P5}, or its equivalent formulations, form a second-order axiomatization of what can be called \emph{classical mereology}. From a certain perspective, this can be called the \emph{standard} approach to mereology. For more on mereology and its meta-properties, see e.g. \citep{Varzi-M,Pietruszczak-M-eng,Cotnoir-Varzi-M}.

\section{Sum structures}\label{sec:axioms-for-sum}
In this part, we present another approach to mereology, with the mereological sum as the primitive, which has not been considered so far in the literature. Given a class of objects $M$, our starting point is a binary relation $\mathord{\+}\subseteq M\times\power(M)$ whose intended interpretation is as follows: $x\+ X$ iff $x$ is a mereological sum of the collection $X$. For the  relation, we assume the following postulates:
\begin{gather}
(\forall X\in\powerne(M))(\exists {x\in X})\,x\+ X\,,\tag{S1}\label{S1}\\
x\+ X\wedge y\+ X\rarrow x=y\,,\tag{S2}\label{S2}\\
x\+ X\wedge y\+ Y\wedge x\in Y\rarrow y\+ X\cup Y\,,\tag{S3}\label{S3}\\
\begin{split}
x\+ X&\wedge x\+ Y\wedge y\in Y\rarrow\\
&(\exists z\in X)(\exists Z,U\in\power(M))\,(z\+ Z\wedge y\+ U\wedge Z\cap U\neq\emptyset)\,.
\end{split}\tag{S4}\label{S4}
\end{gather}

Before showing more axioms for the relation $\+$, let us first explain the ideas of the principles above. Axiom \eqref{S1} states that every non-empty collection has a sum. Principle \eqref{S2} indicates that any collection $X$ of objects has at most one sum. So, \eqref{S1} and \eqref{S2} imply that any non-empty collection of objects has exactly one sum. Next, property \eqref{S3} shows that whenever $y$ is a sum of $Y$ and a collection $X$ of objects has a sum in $Y$,  the object $y$ is a sum of the augmented collection $X\cup Y$. Finally, \eqref{S4} says that if two collections $X,Y$ of objects have the same sum, then for each object in one collection, there is another object in the other collection such that the two objects are sums of collections with some common objects.

The meaning of \eqref{S3} and \eqref{S4} might be made clearer with the help of the following notions of \emph{s-part} and \emph{s-overlapping} relations (we also define \emph{s-disjointness} that will be used later):
\begin{align*}
x\ingrd y&{}\iffdef(\exists X\in\power(M))\,(y\+ X\wedge x\in X)\,,\tag{$\mathrm{df}\ingrd$}\label{df:ingr}
\\
x\soverl y&{}\iffdef(\exists X,Y\in\power(M))\,(x\+ X\wedge y\+ Y\wedge X\cap Y\neq\emptyset)\,,\tag{$\mathrm{df}\,{\soverl}$}\label{df:overls}\\
x\sext y&{}\iffdef(\forall X,Y\in\power(M))\,(x\+ X\wedge y\+ Y\rarrow X\cap Y=\emptyset)\,.\tag{$\mathrm{df}\,\mathord{\sext}$}\label{df:exts}
\end{align*}
\eqref{df:ingr} is a formal embodiment of the Le\'sniewski's idea that $x$ is part of $y$ iff there is a collection $X$ of objects such that $x$ is among $X$-es, and $y$ is a collective class of $X$-es.
The intended interpretation behind $\soverl$ is clear: $x$ and $y$ s-overlap if they sum collections that share at least one object. With the both definitions, \eqref{S4} can be equivalently reformulated in the following  more perspicuous way:
\begin{equation}
    x\+ X\wedge y\ingrs x \rarrow(\exists z\in X)\,y\overls z\,,\tag{\ref{S4}$^\circ$}\label{S4-var}
\end{equation}
which can be interpreted as follows: if $x$ sums $X$ and $y$ is an s-part of $x$, then in $X$ there is a $z$ that s-overlaps $y$.

Let us observe that
\begin{theorem}\label{th:ant-trans}
    If $\langle M,\mathord{\+}\rangle$ satisfies \eqref{S2}--\eqref{S3}, then $\ingrs$ is anti-symmetric and transitive.
\end{theorem}
\begin{proof}
For anti-symmetry, assume that $x\ingrs y$ and $y\ingrs x$.  Thus there are sets $Y$ and $X$ such that $x\+ X$ and $y\in X$, and $y\+Y$ and $x\in Y$. By \eqref{S3} it is the case that $x\+ X\cup Y$ and $y\+X\cup Y$, so $x=y$ by \eqref{S2}.

For transitivity, let $x\ingrs y$ and $y\ingrs z$. So, by \eqref{df:ingr} there are sets $Y$ and $Z$ such that $y\+ Y$ and $x\in Y$, and $z\+ Z$ and $y\in Z$. By \eqref{S3} we get that $z\+ Z\cup Y$, and since $x\in Z\cup Y$ we get that $x\ingrs z$, which ends the proof.
\end{proof}

\subsection{The fifth axiom for \texorpdfstring{$\+$}{}}

Let us now introduce our final axiom for $\+$. For any $x\in M$ let:
\begin{equation}\tag{$\mathrm{df}\,\Ingrs$}\label{df:Ingrs}
\Ingrs(x)\defeq\{y\in M\mid y\ingrs x\}
\end{equation}
and   for any subset $A$ of the domain $M$ let:
\[
\Ingrs(A)\defeq\bigcup_{a\in A}\Ingrs(a)\,.
\]
Let us call a set $A\in\power(M)$ \emph{pre-dense} in a set $B\in\power(M)$ iff for every $b\in B$ there is $a\in A$ such that $a\soverl b$. Thus, our next axiom is:
\begin{equation}\tag{S5}\label{S5}
\text{$X$ is pre-dense in $\Ingrs(x)$}\rarrow x\+\Ingrs(x)\cap\Ingrs(X)\,.
\end{equation}
stating that if every $s$-part of $x$ $s$-overlaps some object of $X$, then $x$ is a sum of the collection consisting of the common $s$-parts of objects in $X$ and the $s$-parts of $x$. On a less formal note, if $X$ is pre-dense in $\Ingrs(x)$, then $\Ingrs(X)$ has enough s-parts of $x$ to add up to~$x$.

% Observe as well that the axiom is well-formulated, as in the case $X$ is pre-dense in $\Ingrs(x)$, the intersection $\Ingrs(x)\cap\Ingrs(X)$ is not empty, since $\Ingrs(x)$ is always non-empty.

We write $\mathbf{S}$ for the class of structures $\langle M,\mathord{\+}\rangle$  satisfying all the axioms \eqref{S1}-\eqref{S5}. In the remainder of the article, we will mainly work with the class ${\bf S}$, and as we shall see, the axioms are desirable in that they are a precise characterization of classical mereology, but now with a new primitive $\+$.

\begin{lemma}\label{lem:sum-of-own-parts}
    If $\langle M,\+\rangle\in\bfS$, then for every $y\in M$, $\{y\}$ is pre-dense in $\Ingrs(y)$, so for every $y\in M$, $y\+\Ingrs(y)$.
\end{lemma}
\begin{proof}
    Let $a\in\Ingrs(y)$, i.e., there is $A\in\power(M)$ such that $y\+ A$ and $a\in A$. By \eqref{S4} there is $b\in A$ such that $\Ingrs(b)\cap\Ingrs(a)\neq\emptyset$, and so $\Ingrs(a)\neq\emptyset$. But by Theorem~\ref{th:ant-trans} it is the case that $\Ingrs(a)\subseteq\Ingrs(y)$, so $\Ingrs(a)\cap\Ingrs(y)\neq\emptyset$. This shows that $\{y\}$ is pre-dense in $\Ingrs(y)$. Applying now \eqref{S5} we get that $y\+\Ingrs(y)\cap\Ingrs(\{y\})$, which means that $y\+\Ingrs(y)$.
\end{proof}
With the help of the lemma we can prove a condition that is closely related to the well-known \emph{weak supplementation principle}.\footnote{See \citep[Lemma 4.1.(iii)]{Pietruszczak-M-eng}.}
\begin{theorem}\label{th:WSP}
    If $\langle M,\+\rangle\in\bfS$, then\/\textup{:}
    \begin{equation}\label{eq:sum-singleton}
        x\+\{y\}\rarrow x=y\,.
    \end{equation}
\end{theorem}
\begin{proof}
Suppose $x\+\{y\}$. If $z\in\Ingrs(x)$, then by \eqref{S4} it is the case that $z\overls y$, so $\{y\}$ is pre-dense in $\Ingrs(x)$. Thus, by \eqref{S5} $x\+\Ingrs(x)\cap\Ingrs(y)$. However, by the assumption $y\ingrs x$, so by the transitivity of $\ingrs$ we have that $\Ingrs(y)\subseteq\Ingrs(x)$, and so $x\+\Ingrs(y)$. By Lemma~\ref{lem:sum-of-own-parts} it is the case that $y\+\Ingrs(y)$, and so by \eqref{S2} $x=y$.
\end{proof}

Since for any $x\in M$ by \eqref{S1} there is a $y$ such that $y\+\{x\}$, by Theorem~\ref{th:WSP} we get that:
\begin{equation}\label{eq:S-singleton}
x\+ \{x\}\,,
\end{equation}
and so it follows that
\begin{theorem}\label{th:partial-order}
    The relation $\ingrs$ is reflexive, and so it is a~partial order by Theorem~\ref{th:ant-trans}.
\end{theorem}
We also obtain that
\begin{equation}\label{eq:Ingrs-not-empty}
(\forall x\in M)\,\Ingrs(x)\neq\emptyset\,.
\end{equation}
\begin{corollary}
    If $\langle M,\+\rangle\in\bfS$, then for no $x\in M$, $x\+\emptyset$.
\end{corollary}
\begin{proof}
    Assume that $x\+ \emptyset$. By \eqref{eq:S-singleton} we have that $x\+\{x\}$, and so \eqref{S4} entails that there is a $z\in\emptyset$, a contradiction.
\end{proof}

\begin{lemma}\label{lem:aux-for-next}
    For any $\langle M,\mathord{\+}\rangle\in\mathbf{S}$\/\textup{:}
    \begin{enumerate}[(1)]
        \item\label{item-1-lem:aux-for-next} $X\subseteq\Ingrs(x)\rarrow\Ingrs(X)\subseteq\Ingrs(x)$\,,
        \item $x\+ X\rarrow x\+\Ingrs(x)\cap\Ingrs(X)$\,,
        \item $x\+ X\rarrow\Ingrs(X)\subseteq\Ingrs(x)$\,,
        \item\label{item-4-lem:aux-for-next} $x\+X\iff x\+\Ingrs(X)$\,.
    \end{enumerate}
\end{lemma}
\begin{proof}
    (1) Let $X\subseteq\Ingrs(x)$ and pick $a\in\Ingrs(X)$. So there is a $x_0\in X$ such that $a\ingrs x_0$. By the assumption, we have that $x_0\ingrs x$, so also $a\ingrs x$ by Theorem~\ref{th:ant-trans}. Thus $a\in\Ingrs(x)$.

    \smallskip

    (2)  Assume that $x\+ X$. By \eqref{eq:Ingrs-not-empty} we have that $\Ingrs(x)\neq\emptyset$. Pick an arbitrary element $y$ from $\Ingrs(x)$, i.e., there is $Y\in\power(M)$ such that $x\+Y$ and $y\in Y$. Applying \eqref{S4-var} to $x\+ X$, $x\+Y$ and $y\in Y$, we conclude that there is a $u\in X$ such that $u\overls y$, so $X$ is pre-dense in $\Ingrs(x)$. By \eqref{S5} it holds that $x\+\Ingrs(x)\cap\Ingrs(X)$, as required.

    \smallskip

    (3) Let $x\+ X$ and $y\in\Ingrs(X)$. So there is $x_0\in X$ such that $y\ingrs x_0$. But $x_0\ingrs x$, so by Theorem~\ref{th:ant-trans} we get that $y\in\Ingrs(x)$.

    \smallskip

    (4) ($\rarrow$) This part follows from (2) and (3).

 \smallskip

($\larrow$) Assume that $x\+\Ingrs(X)$. Thus $\Ingrs(X)\neq\emptyset$, and so $X\neq\emptyset$. Therefore by \eqref{S1} there is a $x_0\in M$ such that $x_0\+ X$. From the first part of the proof of this point, we get that $x_0\+\Ingrs(X)$, and so $x=x_0$ by \eqref{S2}.
\end{proof}

Observe that \eqref{S5} and Lemma \ref{lem:aux-for-next} \eqref{item-1-lem:aux-for-next} imply that:
\begin{equation}\label{eq:S5-var}
\text{$X$ is pre-dense in $\Ingrs(x)$}\wedge X\subseteq\Ingrs(x)\rarrow x\+ X\,.
\end{equation}

\subsection{The sigma operation}\label{sub:sigma}

Given a sum structure $\langle M,\mathord{\+}\rangle$ with every $x\in M$ we can associate the family of all collections $X$ such that $x$ sums $X$. Formally, we define an operation $\Sigmas\colon M\to\power(\power(M))$ such that
\begin{equation}\tag{$\mathrm{df}\,\Sigmas$}
    \Sigmas(x)\defeq\{X\in\power(M)\mid x\+ X\}\,.
\end{equation}

With the axioms introduced so far we can prove that there is no sum of the empty collection:
\begin{equation}\label{eq:no-sum-of-empty}
    (\nexists x\in M)\,x\+\emptyset\qquad\text{so}\qquad(\nexists x\in M)\,\emptyset\in\Sigmas(x)\,.
\end{equation}

It holds that:
\begin{theorem}\label{th:partition}
For any $\langle M,\+\rangle\in\bfS$, the family $\{\Sigmas(x)\mid x\in M\}$ is a partition of $\powerne(M)$.
\end{theorem}
\begin{proof}
Firstly, \eqref{eq:S-singleton} entails that
\begin{equation}\label{eq:Sigmas-non-empty}
(\forall x\in M)\,\{x\}\in\Sigmas(x)\qquad\text{and so}\qquad(\forall x\in M)\,\Sigmas(x)\neq\emptyset\,.
\end{equation}

Secondly, by \eqref{S1} and \eqref{eq:no-sum-of-empty} we have that
    \begin{equation}
        \powerne(M)=\bigcup_{x\in M}\Sigmas(x)\,.
    \end{equation}

Thirdly, by \eqref{S2} we have that
    \begin{equation}
        \Sigmas(x)\cap\Sigmas(y)\neq\emptyset\rarrow\Sigmas(x)=\Sigmas(y)\,.\qedhere
    \end{equation}
\end{proof}

It is routine to verify that:
\[
\Ingrs(x)=\bigcup\Sigmas(x)\,.
\]

Intuitively, if we have a family of sets such that $x\in M$ is the sum of each set from the family, then taking all elements from the sets together should give us nothing more than~$x$:
\begin{equation}\tag{S$\Sigma$}\label{SS}
    \emptyset\neq\calA\subseteq\Sigmas(x)\rarrow\bigcup\calA\in\Sigmas(x)\,.
\end{equation}
Indeed, we have
\begin{theorem}
    \eqref{SS} holds in every sum structure.
\end{theorem}
\begin{proof}
Assume that $\emptyset\neq\calA\subseteq\Sigmas(x)$.   If $a\in \bigcup\calA$, then there is a set $A\in\calA$ such that $a\in A$. Since $x\+A$, we have that $a\in\Ingrs(x)$. So (a) $\bigcup\calA\subseteq\Ingrs(x)$.

Next, let $z\in \Ingrs(x)$. Then, there is some $Z\in\power(M)$ such that $x\+Z$ and $z\in Z$. Note that $\calA\not=\emptyset$, and let $Y\in\calA$. Therefore, $x\+Y$. Applying \eqref{S4} to $x\+Y$, $x\+Z$ and $z\in Z$ we conclude that there is some $y\in Y$ such that $\Ingrs(y)\cap\Ingrs(z)\not=\emptyset$. It is easy to see that $y\in\bigcup\calA$. Therefore, $\bigcup\calA$ is is pre-dense in $\Ingrs(x)$ and so by (a) and \eqref{eq:S5-var} we obtain that $x\+\bigcup\calA$.
\end{proof}

\eqref{SS} is an elegant and <<natural>> postulate about the sum relation. However, it is strictly weaker than \eqref{S5}, as can be seen in Figure~\ref{fig:S5}:\footnote{Diagrams for the sum relation are to be interpreted in the following obvious way: a dashed arrow from an object $x$ to a set $X$ indicates that $x$ is a sum of $X$.} it is simple to check that  $\{b\}$ is pre-dense in $\Ingrs(a)$ and that $\Ingrs(a)\cap\Ingrs(b)=\{b\}$, but $a\+\{b\}$ fails, which violates the axiom.  Even \eqref{SS} and \eqref{eq:sum-singleton} together cannot replace \eqref{S5} as can be seen in the same model.

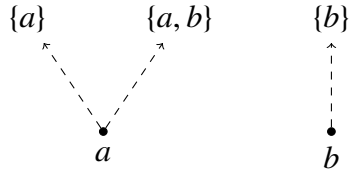
\begin{figure}
    \centering
\begin{tikzpicture}
\node (a) [circle,draw,inner sep=0pt,minimum size=1mm,fill=black] at (0,0) [label=below:$a$]{};
\node (b) [circle,draw,inner sep=0pt,minimum size=1mm,fill=black] at (3,0) [label=below:$b$]{};
\node (a1) at (-1,1.5) {$\{a\}$};
\node (a2) at (1,1.5) {$\{a,b\}$};
\node (b1) at (3,1.5) {$\{b\}$};
\draw[dashed,->] (a) to (a1);
\draw[dashed,->] (a) to (a2);
\draw[dashed,->] (b) to (b1);
%\node(w5) at (2.5,0.6) {$\cdots$};

\end{tikzpicture}
    \caption{A two-element structure $\langle M,\mathord{\+}\rangle$, in which \eqref{S1}--\eqref{S4}, \eqref{SS} and \eqref{eq:sum-singleton} all hold but \eqref{S5} fails.}
    \label{fig:S5}
\end{figure}

\subsection{The equivalence of derived notions}

By means of \eqref{df:ingr} we introduced a derived notion of an \emph{s-part} that---as we have seen in Theorem~\ref{th:partial-order}---is a partial order.  By means of it, we can define the standard notions of overlapping and disjointness:
\begin{align}
x\overl_{\ingrd} y&\:\iffdef\exists_{z\in M}(z\ingrd x\wedge z\ingrd y)\,,\tag{$\mathrm{df}\,\mathord{\overl_{\ingrd}}$}\label{df:overl_ingrd}\\
x\ext_{\ingrd} y&\:\iffdef\neg\exists_{z\in M}(z\ingrd x\wedge z\ingrd y)\,,\tag{$\mathrm{df}\,\mathord{\ext_{\ingrd}}$}\label{df:ext_ingrd}
\end{align}
and the notion of \emph{mereological sum} with respect to $\ingrs$:
\begin{equation*}
    x\Sums X\iffdef(\forall y\in X)\,y\ingrs x\wedge(\forall z\in M)\,(z\ingrs x\rarrow(\exists y\in X)\,y\overl_{\ingrd}z)\,.
\end{equation*}
The above is of course nothing but \eqref{df:Sum} with $\ingrs$ in place of the unspecified part-of relation.

It is easy to see that:
\begin{fact}\label{fact:overl-via-sum}
For any $\langle M,\mathord{\+}\rangle$, it holds that \/\textup{:}
\begin{align}
    x\soverl y &{}\iff\Ingrs(x)\cap\Ingrs(y)\neq\emptyset\iff x\overl_{\ingrs} y\,,\\
    x\sext y &{}\iff\Ingrs(x)\cap\Ingrs(y)=\emptyset\iff x\ext_{\ingrs} y\,.
\end{align}
\end{fact}
A somewhat harder task is to demonstrate that $\Sums$ coincide with $\+$. Let's get down to work.

\begin{lemma}\label{lem:S->Singr}
    For any $\langle M,\mathord{\+}\rangle$ satisfying \eqref{S4}, if $x\+ X$, then $x\+_{\ingrd} X$.
\end{lemma}
\begin{proof}
    Assume that $x\+ X$. Directly from \eqref{df:ingr} we get that $(\forall y\in X)\,y\ingrs x$. Further, if $y\ingrs x$, then by \eqref{df:ingr} there is a $Y$ such that $y\in Y$ and $x\+ Y$. Therefore by \eqref{S4-var} there is some $z\in X$ such that $z\overls y$, and by Fact~\ref{fact:overl-via-sum} it is the case that $z\overl_{\ingrs} y$, as required. So, by \eqref{df:Sum} it holds that $x\+_{\mathord{\ingrs}} X$.
\end{proof}

Let us now prove the converse of Lemma~\ref{lem:S->Singr}:

\begin{lemma}\label{Singr->S}
    For any $\langle M,\mathord{\+}\rangle\in\mathbf{S}$, if $x\+_{\ingrd} X$, then $x\+ X$.
\end{lemma}

\begin{proof}
 We assume that $x\+_{\ingrd} X$. Then, it holds that
\begin{enumerate}[(a)]
    \item $X\subseteq \Ingrs(x)$, and
    \item for all $a\in M$, if $a\in \Ingrs(x)$, then there is $y\in X$ such that $\Ingrs(y)\cap\Ingrs(a)\not=\emptyset$.
\end{enumerate}
Applying \eqref{S5} to (b) we obtain that $x\+\Ingrs(x)\cap\Ingrs(X)$.
By (a) and Lemma~\ref{lem:aux-for-next} \eqref{item-1-lem:aux-for-next} it is the case that $x\+\Ingrs(X)$. So $x\+ X$ by Lemma~\ref{lem:aux-for-next} \eqref{item-4-lem:aux-for-next}.
\end{proof}

Putting together Lemma \ref{lem:S->Singr} and  Lemma \ref{Singr->S}, we have the following
\begin{theorem}\label{th:sum-and-sum-induced-by-ingr-are-the-same}
For any $\langle M,\mathord{\+}\rangle\in\mathbf{S}$, it holds that\/\textup{:}
\[
x\+ X \iff x\+_{\ingrd} X\,.
\]
\end{theorem}

\subsection{Parthood axioms in the sum setting}

We have already seen that the binary relation $\ingrs$ is a partial order. To show that it deserves the name \emph{parthood} we need to know that both \eqref{P4} and \eqref{P5} hold for it. The latter is an immediate consequence
Theorem~\ref{th:sum-and-sum-induced-by-ingr-are-the-same} and \eqref{S1}. For the former we use the following:
\begin{theorem}[{\citealp[Theorem 6.1]{Pietruszczak-M-eng}}]
    The condition \eqref{P4} holds in all structures that satisfy \eqref{P1}--\eqref{P3}, \eqref{P5} and\/\textup{:}
    \[
        x\Sumd X\wedge y\Sumd X\rarrow x=y\,.
    \]
\end{theorem}
But the above axiom holds for $\ingrs$ due to \eqref{S2} and Theorem~\ref{th:sum-and-sum-induced-by-ingr-are-the-same}, and thus we have that:
\begin{equation}\tag{P4$'$}
    x\ningrs y\rarrow (\exists z\in M)\,(z\ingrs x\wedge z\ext_{\ingrs} y)
\end{equation}
obtains in any structure from $\mathbf{S}$. Therefore
\begin{theorem}\label{th:ingr-induced-from-sum-satisfies-ingr-axioms}
If $\langle M,\mathsf{S}\rangle\in\mathbf{S}$ and $\langle M,\ingrd,\mathsf{S}\rangle$ is its definitional extension by \eqref{df:ingr},
then $\langle M,\mathord{\ingrd},\mathsf{S}\rangle$ satisfies all $\ingr$-axioms.
\end{theorem}

\section{The sum axioms in mereological structures}\label{sec:sum-axioms-in-MS}

To show the correctness of our approach, it remains to demonstrate that
\begin{theorem}\label{th:sum-induced-from-ingr-satisfies-sum-axioms}
If $\langle M,\mathord{\ingr}\rangle\in\MS$ and $\langle M,\mathord{\ingr},\mathord{\Sumd}\rangle$ is its definitional extension by \eqref{df:Sum}, then $\langle M,\mathord{\ingr},\mathord{\Sumd}\rangle$ satisfies all $\+$-axioms.
\end{theorem}
\begin{proof}
    \eqref{S1} Immediate from \eqref{P5} and \eqref{df:Sum}.

    \smallskip

    \eqref{S2} Let $x\Sumd X$ and $y\Sumd X$. If $u\ingr x$, then by \eqref{df:Sum} there is an $a\in X$ such that $a\poverl u$. But, again by the definition, $a\ingr y$, so by the arbitrariness of $u$ and by \eqref{P4}, we obtain that $x\ingr y$. In the analogous way we prove that $y\ingr x$, so from the antisymmetry of parthood, we obtain that $x=y$.

    \smallskip

    \eqref{S3} Assume $x\Sumd X$, $y\Sumd Y$ and $x\in Y$. If $a\in X$, then $a\ingr x$. But $x\ingr y$ by \eqref{df:Sum}, so transitivity of parthood entails that $a\ingr y$, as required. If $a\in Y$, then $a\ingr y$ from the same definition and from the second assumption. Finally, if $a\ingr y$, then there is an $a\in Y$ (and the more so in $X\cup Y$) such that $a\poverl y$.

    \smallskip

    \eqref{S4} Suppose $x\Sumd X$ and $x\Sumd Y$ and $y\in Y$. Since $y\ingr x$, in $X$ there is a $u$ such that $y\poverl u$. Put $Z\defeq\{w\in M\mid w\ingr y\}$ and $U\defeq\{w\in M\mid w\ingr u\}$. Clearly, $Z\cap U\neq\emptyset$ and $y$ and $u$ are sums of, respectively, $Z$ and $U$.

    \smallskip

    \eqref{S5} Observe that for $\Sumd$, $\Ingr_{\Sumd}(x)=\{y\in M\mid y\ingr x\}$ (recall \eqref{eq:equality-of-parthoods}). Let $X$ be pre-dense in $\Ingr_{\Sumd}(x)$. It is clear that every element of $\Ingr_{\Sumd}(x)\cap\Ingr_{\Sumd}(X)$ must be part of~$x$. On the other hand, if $y\ingr x$, then in $X$ there is an $a$ such that $a\poverl y$. Assume that $b$ is one of their common parts, i.e., $b\ingr a$ and $b\ingr y$. Then, for this $b$, by the reflexivity of $\ingr$ it holds  that $b\ingr b$, and so $b\poverl y$. Moreover, it follows from the transitivity of $\ingr$ that the object $b$ is an element of $\Ingr_{\Sumd}(x)\cap\Ingr_{\Sumd}(X)$, so \eqref{S5} holds for $\Sumd$.

    \smallskip

    All the points above together show that any definitional extension of a~mereological structure by means of \eqref{df:Sum} satisfies the axioms for the mereological sum.
\end{proof}

\section{Independence of the axioms for \texorpdfstring{$\+$}{}}\label{sec:independence}

We will now show that

\begin{theorem}
\eqref{S1}--\eqref{S5} is an independent set of axioms.
\end{theorem}
\begin{proof}
    \eqref{S1} Consider the structure $\langle M,\mathord{\+}\rangle$, where $M\defeq\{a,b\}$ and $\+$ consists of $a\+\{a\}$ and $b\+\{b\}$. So, the collection $\{a,b\}$ does not have any sum, which indicates that the principle \eqref{S1} fails. However, it is routine to check that all the remaining axioms hold for the structure.

    \smallskip

    \eqref{S2} For a witness, consider again the set $M\defeq\{a,b\}$ and the structure depicted in Figure \ref{fig:independence-S2}. We have $a\+\{a,b\}$ and $b\+\{a,b\}$, but $a\neq b$. So \eqref{S2} fails. But all other axioms are satisfied, as needed.

    \begin{figure}
    \centering
\begin{tikzpicture}
\node (a) [circle,draw,inner sep=0pt,minimum size=1mm,fill=black] at (0,0) [label=below:$a$]{};
\node (b) [circle,draw,inner sep=0pt,minimum size=1mm,fill=black] at (3,0) [label=below:$b$]{};
\node (a1) [] at (0,1.5) {$\{a\}$};
 \node(b1)[] at (3,1.5) {$\{b\}$};
  \node(c)[] at (1.5,1.5) {$\{a,b\}$};
\draw[dashed,->](a) to (a1);
\draw[dashed,->](a) to (c);
 \draw[dashed,->](b) to (b1);
  \draw[dashed,->](b) to (c);
%\node(w5) at (2.5,0.6) {$\cdots$};
\end{tikzpicture}
    \caption{A structure $\langle M,\mathord{\+}\rangle$ where \eqref{S2} fails.}
    \label{fig:independence-S2}
\end{figure}
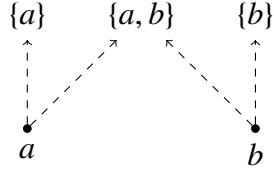

    \smallskip

    \eqref{S3} To make the presentation clear, we will construct a suitable model from a non-transitive part-of relation depicted in Figure~\ref{fig:for-S3}. The diagram is to be interpreted bottom-up, e.g., $4$ is part of $2$, and so on. $x$ is part of $y$ only if there is  bottom-up path from $x$ and $y$ composed of lines of the same style. Thus, we see, for example, that $6$ is part of $5$ and $5$ is part of $3$, but $6$ is not a part of $3$. So the relation $\ingr$ in the diagram is not transitive. Let $\Sumd$ be the sum relation for $\ingr$, and let us consider $\langle M,\Sumd\rangle$. As it can be seen, $3\Sumd\{4,5\}$ and $5\Sumd\{6,7\}$ but not $3\Sumd\{4,5,6,7\}$, as $6\ningr 3$. We leave the straightforward verification of the remaining four axioms to the reader.

    \begin{figure}
        \centering
        \begin{tikzpicture}
            \coordinate (4) at (0,0);
            \coordinate (6) at (2,0);
            \coordinate (7) at (4,0);
            \coordinate (5) at (3,1.5);
            \coordinate (2) at (0,3);
            \coordinate (3) at (3,3);
            \coordinate (1) at (1.5,4.5);
            \node [circle,draw,inner sep=0pt,minimum size=1mm,fill=black,label=below:$4$] at (4) {};
            \node [circle,draw,inner sep=0pt,minimum size=1mm,fill=black,label=below:$6$] at (6) {};
            \node [circle,draw,inner sep=0pt,minimum size=1mm,fill=black,label=below:$7$] at (7) {};
            \node [circle,draw,inner sep=0pt,minimum size=1mm,fill=black,label=right:$5$] at (5) {};
            \node [circle,draw,inner sep=0pt,minimum size=1mm,fill=black,label=right:$3$] at (3) {};
            \node [circle,draw,inner sep=0pt,minimum size=1mm,fill=black,label=left:$2$] at (2) {};
            \node [circle,draw,inner sep=0pt,minimum size=1mm,fill=black,label=above:$1$] at (1) {};
            \draw (4) -- (2) (2) -- (1) (3) -- (1) (5) -- (3) (7) -- (5) (6) -- (2) (4) -- (3);
            \draw [dotted] (6) -- (5);
        \end{tikzpicture}
        \caption{A parthood basis for the sum structure in which \eqref{S3} fails.}
        \label{fig:for-S3}
    \end{figure}
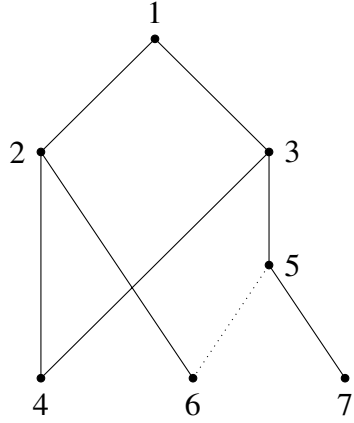

    \smallskip

    \eqref{S4} Consider the structure provided in Figure \ref{fig:independence-S4}: $M\defeq\{a\}$ and $a\+\{a\}$ and $a\+\emptyset$. The axiom \eqref{S4} fails in it, as there are no objects in the empty set. It can be easily seen that all the remaining four axioms are satisfied.

\begin{figure}
    \centering
\begin{tikzpicture}
\node(a)[circle,draw,inner sep=0pt,minimum size=1mm,fill=black] at (0,0) [label=below:$a$]{};
\node(a1)[] at (-1,1.5) {$\{a\}$};
 \node(a2)[] at (1,1.5) {$\emptyset$};
\draw[dashed,->](a) to (a1);
\draw[dashed,->](a) to (a2);
\end{tikzpicture}
    \caption{A structure $\langle M,\mathord{\+}\rangle$ where \eqref{S4} fails.}
    \label{fig:independence-S4}
\end{figure}
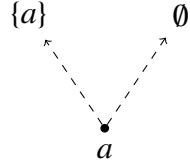

    \smallskip

    \eqref{S5} A witness to this is the structure given in Figure \ref{fig:S5}.

        \smallskip

        Therefore, any axiom is not a consequence of other axioms, and the set of all axioms is independent.
\end{proof}

\section{Summary and further work}\label{sec:summary}

We have presented and analyzed well-justified axioms for a system of mereology in which the \emph{sum} relation replaces the \emph{part-of} relation as a basic concept. The system composed of \eqref{S1}--\eqref{S5} works well in the sense that within it we capture the binary \emph{parthood} relation, and we prove that there is a one-to-one correspondence between mereological structures (in the sense of \citealp{Pietruszczak-M-eng})  and the sum structures as defined here.

There are good things and there is one we consider <<bad>> about our system. The good things are axioms \eqref{S1}--\eqref{S4} that are natural, simple and has an immediate interpretation. \eqref{S5}, on the other hand,  formulated in the two primitives of the theory, $\+$ and $\in$, is  complex, convoluted and contrasts with the simplicity of the remaining four postulates:
\begin{footnotesize}
\begin{equation*}
\begin{split}
(\forall y\in M)\,\bigl[(\exists Z\in\power(M))(x\+ Z&{}\wedge  y\in Z)\rarrow (\exists u\in X)(\exists Y, U\in\power(M)) (y\+ Y\wedge u\+ U\wedge Y\cap Z\not=\emptyset)\bigr]\rarrow \\
&x\+ \bigcup_{u\in X}\{y\in M\mid (\exists U_1,U_2\in\power(M))\,  (x\+ U_1\wedge u\+ U_2\wedge y\in U_1\wedge y\in U_2)  \}\,.
\end{split}
\end{equation*}
\end{footnotesize}

%\text{put here \eqref{S5} couched in primitive terms}

%$$(\forall y\in \Ingrs(x))(\exists u\in X) u\overls y \to x\+ \Ingrs(x)\cap\Ingrs(X) $$

\noindent We believe that it can be replaced by simpler conditions, which we hope to discover in future investigations into the sum structures. We also believe that the sigma operation from Section~\ref{sub:sigma} may be employed to obtain a completely new abstract perspective on theories of parts and wholes.

The theory from this paper falls within the scope of second-order systems. There are also other formal tools that are suitable for studying different relations, among which modal logic is an important tradition. As suggested by \cite{Blackburn-et-al-ML}, the tool usually has lower computational complexity although it is essentially a second-order theory, while many theories of mereology are undecidable (see e.g., \citealp{Tsai-ACPOTHOMT}). Based on a {\em relational semantics}, \cite{Li-et-al-MBL} defined a series of mereological theories for $\ingr$, but it remains to develop a desired modal system that corresponds to our axioms for $\+$, for which the {\em neighborhood semantics} \citep{Pacuit-NSFML} would be useful. The development of this kind of approach is going to be our next step.

\section*{Acknowledgements}

\begin{sloppypar}

This research was funded by the National Science Center (Poland), grant number 2020/39/B/HS1/00216, ``Logico-philosophical foundations of geometry and topology''.

For the purpose of Open Access, the authors have applied a CC-BY public copyright license to any Author Accepted Manuscript (AAM) version arising from this submission.

\end{sloppypar}

\bibliographystyle{plainnat}
\providecommand{\noop}[1]{}

\end{document}